\documentclass{amsart}
\usepackage{textcomp}
\usepackage{amssymb}
\usepackage{verbatim}
\usepackage{array}
\usepackage{latexsym}
\usepackage{enumerate}
\usepackage{amsmath}
\usepackage{amsfonts}
\usepackage{amsthm}
\usepackage{color}
\usepackage[english]{babel}
\usepackage[notref,notcite]{showkeys}
\usepackage{comment}
\usepackage{units}

\newtheorem{theorem}{Theorem}[section]

\newtheorem{lemma}[theorem]{Lemma}

\newtheorem{result}[theorem]{Result}

\newtheorem{rem}[theorem]{Remark}

\definecolor{mybro}{RGB}{35, 60, 200}

\def\cC{\mathcal C}
\def\cD{\mathcal D}

\def\cF{\mathcal F}

\def\cH{\mathcal H}

\def\cX{\mathcal X}
\def\cY{\mathcal Y}

\def\PG{{\rm{PG}}}

\def\deg{\mbox{\rm deg}}

\def\div{\mbox{\rm div}}

\def\fq{{\mathbb F}_q}
\def\Fs{{\mathbb F}_{q^2}}

\newcommand{\PGL}{\mbox{\rm PGL}}

\newcommand{\PGU}{\mbox{\rm PGU}}

\newcommand{\aut}{\mbox{\rm Aut}}

\newcommand{\go}{\omega}

\newcommand{\ha}{{\textstyle\frac{1}{2}}}

\title[Galois subcovers  of the Hermitian curve]{Galois subcovers  of the Hermitian curve in characteristic $p$ with respect to subgroups of order $p^2$}
\begin{document}

\author{Barbara Gatti , G\'abor Korchm\'aros}

\begin{abstract}
A (projective, geometrically irreducible, non-singular) curve $\mathcal{X}$ defined over a finite field $\mathbb{F}_{q^2}$ is \emph{maximal} if the number $N_{q^2}$ of its $\mathbb{F}_{q^2}$-rational points attains the Hasse-Weil upper bound, that is $N_{q^2}=q^2+2\mathfrak{g}q+1$ where $\mathfrak{g}$ is the genus
of $\mathcal{X}$. An important question, also motivated by applications to algebraic-geometry codes, is to find explicit equations for maximal curves. For a few curves which are Galois covered of the Hermitian curve, this has been done so far ad hoc, in particular  in the cases where the Galois group has prime order. In this paper we obtain explicit equations of all Galois covers of the Hermitian curve with Galois group of order $p^2$ where $p$ is the characteristic of $\mathbb{F}_{q^2}$. Doing so we also determine the $\mathbb{F}_{q^2}$-isomorphism classes of such curves and describe their full $\mathbb{F}_{q^2}$-automorphism groups.  
\end{abstract}

\maketitle
\vspace{0.5cm}\noindent {\em Keywords}:
maximal curves, function fields, Galois cover
\vspace{0.2cm}\noindent

\vspace{0.5cm}\noindent {\em Subject classifications}:
\vspace{0.2cm}\noindent  14H37, 14H05.

\vspace{0.5cm}\noindent {\em Essential informations}:
Barbara Gatti (Corresponding author)\\ 
barbara.gatti@unisalento.it\\
Department of Mathematics and Physics, University of Salento "Ennio de Giorgi", Lecce, Italy \\

\vspace{0.2cm}\noindent 
G\'abor Korchm\'aros\\ 
gabor.korchmaros@unibas.it\\
Department of Mathematics, Computer Science and Economics, University of the Basilicata,
Potenza, Italy

\section{Introduction}
It is well known that curves with many points over a finite field have several interesting features, both theoretical and applicative. They have been intensively studied since 1980s, also in connection with Coding theory,  Cryptography, Finite geometry, and shift register sequences.  The most important and yet better understood family consists of the maximal curves, that is, curves defined over $\mathbb{F}_{q^2}$ which attain the Hasse-Weil upper bound. As it has emerged from works of Serre, Garc\'ia, van der Geer, Stichtenoth and others, the main questions about maximal curves are;
\begin{itemize}
\item[(i)] Determination of the possible genera of maximal curves over a given finite field;
\item[(ii)] Determination of explicit equations for maximal curves;
\item[(iii)] Classification of maximal curves over a given finite field which have the same genus.
\end{itemize}
Classical examples of maximal curves are the $1$-dimensional Deligne-Lusztig varieties, namely the Hermitian curve, the Suzuki curve (in characteristic $2$) and the Ree curve (in characteristic $3$). By force of a result attributed to Serre, they give rise many other maximal curves since any $\mathbb{F}_{q^2}$-subcover of a maximal curve remains maximal over the same field. Particular  $\mathbb{F}_{q^2}$-subcovers are the Galois subcovers with respect to the $\mathbb{F}_{q^2}$-automorphism groups in which case the resulting curves are named quotient curves. Since the $\mathbb{F}_{q^2}$-automorphism group of the Hermitian curve is isomorphic to the $3$-dimensional projective unitary group $\PGU(3,q)$ which is rich of subgroups, the family of quotient curves arising from the Hermitian curve is large. Question (i) for quotient curves of the Hermitian curve has intensively been investigated following the seminal paper of Garc\'ia, Stichtenoth and Xing \cite{GSX}.  A relevant result is the complete solution of Question (i) for $q\equiv 1 \pmod 4$ which has been achieved in a series of papers by Montanucci and Zini. The other case $q\equiv 3 \pmod 4$ is still under investigation.

Our contributes concern Questions (ii) and (iii). Families of curves defined by explicit equations which also include maximal curves are found in the literature, see \cite[Section 10.1, 10.8, 10.9]{HKT}, and \cite{BGKM}, \cite{BLSY}, \cite{cossidente-korchmaros-torres1999}, \cite{mq}, \cite{TS}, \cite{TT2}, \cite{TT3}. However, a systematic study of Question (ii) has been done so fare only for quotient curves with respect to $\mathbb{F}_{q^2}$-automorphism groups of prime order; see \cite{cossidente-korchmaros-torres2000}, \cite[Theorem 12.28]{HKT}.
In this paper, we completely solve Question (ii) for the family of quotient curves with respect to $\mathbb{F}_{q^2}$-automorphism groups of order $p^2$ where $p$ is the characteristic of  $\mathbb{F}_{q^2}$.
\begin{theorem}
\label{th1} In the $\mathbb{F}_{q^2}$-automorphism group $G\cong \PGU(3,q)$  of the Hermitian curve $\cH_q$ defined over $\mathbb{F}_{q^2}$ with $q=p^h$, let $\Phi$ be a subgroup of order $p^2$.
\begin{itemize}
\item[(I)] If $\Phi$ is an elementary abelian and contained in the center of the Sylow $p$-subgroup of $G$ containing $\Phi$, then the quotient curve $\cH_q/\Phi$ has equation
\begin{equation}
\label{eqthI}
\sum_{i=1}^{h-1}(b-b^{p^i})X^{p^{i-1}}+\go Y^{q+1}=0
\end{equation}
for $b\in \mathbb{F}_q \setminus \mathbb{F}_p$ where $\go\in \mathbb{F}_{q^2}$ is a fixed element such that $\go^{q-1}=-1$.
\item[(II)] If $\Phi$ is an elementary abelian and it is not contained in the center of the Sylow $p$-subgroup of $G$ containing $\Phi$, then $p>2$ and the quotient curve $\cH_q/\Phi$ has equation
\begin{equation}
\label{eqthII}
\Big(\sum_{i=1}^{h}X^{p^{i-1}}\Big)^2-2b\sum_{i=1}^{h}Y^{p^{i-1}}=0
\end{equation}
where $b^q+b=0$.
\item[(III)] If $\Phi$ is cyclic then $p=2$ and the quotient curve $\cH_q/\Phi$ has equation
\begin{equation}
\label{eqthIII}
\alpha_0(X)+\alpha_1(X)Y+\ldots+\alpha_i(X)Y^{2^i}+\ldots+ \alpha_{h-1}(X)Y^{2^{h-1}}=0
\end{equation}
where $\alpha_i(X)\in \mathbb{F}_{q^2}[X]$ and
$$\alpha_0(X)=X^{q+1},\,\, \alpha_1(X)= \frac{(X^q+X)^2+(X+b+b^2)^q(X^q+X)}{Tr(X)},\,\,\alpha_h(X)=(X+b+b^2)^{2q}$$
and the other coefficients $\alpha_i(X)$ for $2\le i \le h-2$ are computed recursively from the equation $$Tr(X)\alpha_i(X)+\alpha_{i-1}(X)^2+(X+b+b^2)^{2q}+(X+b+b^2)^q(X^q+X)$$
where $Tr(X)=X+X^p+\ldots+X^{\nicefrac{q}{p}}$.
\end{itemize}
\end{theorem}
Question (iii) in its generality appears to be the most difficult among the above three questions because of lack of an adequate approach which may make it possible to compare maximal curves over the same field whose main birational invariants coincide, such as genera,  automorphism groups, and Weierstrass semigroups. It is already challenging enough to find two non-isomorphic maximal curves over the same field which have the same genus; a few examples are exhibited in \cite{cossidente-hirschfeld-korchmaros-torres2000} and \cite{GHKT}. Actually, the difference between two maximal curves defined over the same finite field may be more subtle. In fact, as a corollary of the following Theorems \ref{thI} and \ref{thII}, there are even non-isomorphic maximal curves over the same field which have the same genus and $\mathbb{F}_{q^2}$-automorphism group, as well as, the same Weierstrass semigroup at some point.
\begin{theorem}
\label{thI} The curve $\cX_b$ defined by Equation (\ref{eqthI}) has the following properties.
\begin{itemize}
\item[(Ia)] $\mathfrak{g}(\cX_b)=\ha q(p^{h-2}-1)=\ha q(\frac{q}{p^2}-1).$
\item[(Ib)] The $\mathbb{F}_{q^2}$-automorphism group of $\cX_b$ is the semidirect product of a normal subgroup of order $p^{h-2}$ by a (cyclic) complement of order $(q+1)(p-1)$.
\item[(Ic)] Let $P_\infty$ be the unique point of $\cX_b$ at infinity. The Weierstrass semigroup at the unique place centered at $P_\infty$ is generated by $p^{h-2}$ and $q+1$.
\item[(Id)] Let $\bar{b}\in \mathbb{F}_q\setminus \mathbb{F}_p$. Then $\cX_b$ and $\cX_{\bar{b}}$ are $\mathbb{F}_{q^2}$-isomorphic in exactly one of the following cases
\begin{itemize}
      \item[(Id1)] $b,\bar{b}\in \mathbb{F}_{p^2}$ or $b,\bar{b}\in \mathbb{F}_{p^3}$;
      \item[(Id2)] $b,\bar{b}\not\in \mathbb{F}_{p^2}\cup \mathbb{F}_{p^3}$ and
$\bar{b}= \frac{\alpha b +\beta}{\gamma b + \delta}$
with $\alpha,\beta,\gamma,\delta \in \mathbb{F}_p$ and $\alpha \delta - \beta \gamma \neq 0$.
\end{itemize}
\item[(Ie)] Assume that $\mathbb{F}_q$ has a proper subfield $\mathbb{F}_{p^m}$ larger than $\mathbb{F}_p$. If $b\in\mathbb{F}_p$ and $\bar{b}\in\mathbb{F}_q\setminus \mathbb{F}_{p^m}$ then $G_b$ and $G_{\bar{b}}$ are not $\mathbb{F}_{q^2}$-isomorphic.
\end{itemize}
\end{theorem}
\begin{theorem}
\label{thII} The curve $\cX_b$ defined by Equation (\ref{eqthII}) has the following properties.
\begin{itemize}
\item[(IIa)] $\mathfrak{g}(\cX_b)=\ha\frac{q}{p}(\frac{q}{p}-1)=\ha p^{h-1}(p^{h-1}-1)$.
\item[(IIb)] The $\mathbb{F}_{q^2}$-automorphism group of $G_b$ is the semidirect product of a normal subgroup of order $\textstyle\frac{q^2}{p}$ by a (cyclic) complement of order $p-1$.
\item[(IIc)] Let $P_\infty$ be the unique point of $G_b$ at infinity. The Weierstrass semigroup at the unique place centered at $P_\infty$ is generated by $\textstyle\frac{q}{p}, \textstyle\frac{q}{p}+\textstyle\frac{q}{p^2}$ and $q+1$.
\item[(IId)] Let $\bar{b}\in \mathbb{F}_q\setminus \mathbb{F}_p$. Then $\cX_b$ and $\cX_{\bar{b}}$ are $\mathbb{F}_{q^2}$-isomorphic if and only if there exists $\kappa\in \mathbb{F}_{p}^*$ such that $\bar{b}=\kappa b$.
 \end{itemize}
\end{theorem}
The curve $\cX$ defined by Equation (\ref{eqthIII}) can also be obtained by fiber product and it is useful for applications in Coding theory; see \cite{vg}. Our contribution consists in determining its full $\mathbb{F}_{q^2}$-automorphism group. 
\begin{theorem}
\label{thIII} The curve $\cX$ defined by Equation (\ref{eqthIII}) has the following properties.
\begin{itemize}
\item[(IIIa)] $\mathfrak{g}(\cX)=\textstyle\frac{1}{8}q(q-2)$
\item[(IIIb)] The $\mathbb{F}_{q^2}$-automorphism group of $\cX$ has order $\ha q^2$ and exponent $4$.
 \end{itemize}
\end{theorem}
Our Notation and terminology are standard; see \cite{HKT,stich1993,huppertI1967, serre1979/87-68
7389}. We mostly use the language of function field theory rather than that of algebraic geometry.

\section{Background}
Let $\cX$ be a projective, non-singular, geometrically irreducible, algebraic curve of genus $\mathfrak{g}\geq 2$ embedded in an $r$-dimensional projective space $\PG(r,\mathbb{F}_\ell)$ over a finite field of order $\ell$ of characteristic $p$.  Let $\mathbb{F}_\ell(\cX)$ be the function field of $\cX$ which is an algebraic function field of transcendency degree one with constant field $\mathbb{F}_\ell$. As it is customary, $\cX$ is viewed as a curve defined over the algebraic closure $\mathbb{F}$ of $\mathbb{F}_\ell$. Then the function field $\mathbb{F}(\cX)$ is the constant field extension of $\mathbb{F}_\ell(\cX)$ with respect to field extension $\mathbb{F}|\mathbb{F}_\ell$. The automorphism group $\aut(\cX)$ of $\cX$ is defined to be the automorphism group of $\mathbb{F}(\cX)$ fixing every element of $\mathbb{F}$. It has a faithful permutation representation on the set of all points $\cX$ (equivalently on the set of all places of $\mathbb{F}(\cX))$. The automorphism group $\aut({\mathbb{F}_\ell}(\cX))$ of $\mathbb{F}_\ell(\cX)$ is a subgroup of $\aut(\cX)$. In particular, the action of $\aut({\mathbb{F}_\ell}(\cX))$ on the $\mathbb{F}_\ell$-rational points of $\cX$ is the same as on the set of degree $1$ places of $\mathbb{F}_\ell(\cX)$.

Let $G$ be a finite subgroup of $\aut({\mathbb{F}_{\ell}}(\cX))$. The \emph{Galois subcover} of $\mathbb{F}_\ell(\cX)$ with respect to $G$ is the fixed field of $G$, that is, the subfield ${\mathbb{F}_{\ell}}(\cX)^G$ consisting of all elements of $\mathbb{F}_{\ell}(\cX)$ fixed by every element in $G$. Let $\cY$ be a non-singular model of ${\mathbb{F}_{\ell}}(\cX)^G$, that is,
a projective, non-singular, geometrically irreducible, algebraic curve with function field ${\mathbb{F}_{\ell}}(\cX)^G$. Then $\cY$ is the \emph{quotient curve of $\cX$ with respect to $G$} and is denoted by $\cX/G$. The covering $\cX\mapsto \cY$ has degree $|G|$ and the field extension $\mathbb{F}_{\ell}(\cX)|\mathbb{F}_{\ell}(\cX)^G$ is Galois.


If $P$ is a point of $\cX$, the stabiliser $G_P$ of $P$ in $G$ is the subgroup of $G$ consisting of all elements
fixing $P$.
\begin{result}\cite[Theorem 11.49(b)]{HKT}
\label{resth11.49b} All $p$-elements of $G_P$ together with the identity form a normal subgroup $S_P$ of $G_P$ so that $G_P=S_P\rtimes C$, the semidirect product of $S_P$ with a cyclic complement $C$.
\end{result}
\begin{result}\cite[Theorem 11.129]{HKT}
\label{resth11.129} If $\cX$ has zero Hasse-Witt invariant then every non-trivial element of order $p$ has a unique fixed point, and hence no non-trivial element in $S_P$ fixes a point other than $P$.
\end{result}
A useful corollary of Result \ref{resth11.129} is the following.
\begin{result}
\label{lem15042023} Let $\cX$ be an  $\mathbb{F}_\ell$-rational curve whose number of $\mathbb{F}_\ell$-rational points is $N\ge 2$. If $\cX$ has zero Hasse-Witt invariant and $S$ is a $p$-subgroup of $\aut({\mathbb{F}_\ell}(\cX))$ then $S$ fixes a unique point and $|S|$ divides $N-1$.
\end{result}
The following result is due to Stichtenoth \cite{stichtenoth1973I}. 
\begin{result} \cite[Theorem 11.78(i)]{HKT}
\label{sti1} Let $S$ be a $p$-subgroup of $\mathbb{F}(\cX)$ fixing a point. If $|S|$ is larger than the genus of $\mathbb{F}(\cX)$ then the Galois subcover of $\mathbb{F}(\cX)$ with respect to $S$ is rational. 
\end{result}
 
 From now on let $\ell=q^2$ with $q=p^h$ and assume that $\cX$ is a $\mathbb{F}_{q^2}$-maximal curve.
\begin{result}
\label{zeroprank} All $\mathbb{F}_{q^2}$-maximal curves have zero Hasse-Witt invariant.
\end{result}
The following result is commonly attributed to Serre, see Lachaud \cite{lachaud1987}.
\begin{result}\cite[Theorem 10.2]{HKT}
\label{resth10.2} For every subgroup $G$ of $\aut({\mathbb{F}_{q^2}}(\cX))$, the quotient curve $\cX/G$ is also  $\mathbb{F}_{q^2}$-maximal.
\end{result}

Let $S$ be a numerical semigroup. The gaps of $S$ are the elements in $\mathbb{N}\setminus S$. The number $g(S)$ of gaps of $S$ is the {\em{genus}} of $S$. If $S$ is the Weierstrass semigroup of a curve at a point then $g(S)$ coincides with the genus of the curve.  Let $(a_1,\ldots,a_k)$ be a sequence of positive integers such that their greatest common divisor is 1. Let $d_0=0,\,d_i=g.c.d.(a_1,\ldots,a_k)$ and $A_i=\{\frac{a_1}{d_i},\ldots,\frac{a_i}{d_i}\}$ for $i=1,\ldots,k$. Let $S_i$ be the semigroup generated by $A_i$. The sequence $(a_1,\ldots,a_k)$ is {\em{telescopic}} whenever $\frac{a_i}{d_i}\in S_{i-1}$ for $i=2,\ldots,k$. A  {\em{telescopic semigroup}} is a numerical semigroup generated by a telescopic sequence.  
\begin{result}\label{resgeneresemigroup}\cite[Lemma 6.5]{KirfelPellikaan1995}
 For the semigroup generated by a telescopic sequence $(a_1,\ldots,a_k)$,  
 $$l_g(S_k)=\sum_{i=1}^{k}(\frac{d_{i-1}}{d_i}-1)a_1,\quad
 g(S_k)=\frac{l_g(S_k)+1}{2}.$$
\end{result}
\subsection{The Hermitian curve and its function field}
 In the remainder of the paper, we focus on the Hermitian function fields and its Galois subcovers. The usual affine equation, or canonical form, of the Hermitian curve $\cH_q$, as an $\mathbb{F}_{q^2}$-rational curve, is $Y^q+Y=X^{q+1}$ and hence its function field is $\mathbb{F}_{q^2}(x,y)$ with $y^q+y-x^{q+1}=0$. Another useful equation of $\cH_q$ is $Y^q-Y+\omega X^{q+1}=0$ with $\omega\in \mathbb{F}_{q^2}$ such that $\omega^{q-1}=-1$. We collect a number of known results on the $\mathbb{F}_{q^2}$-automorphism group $\aut(\mathbb{F}_{q^2}(\cH_q))$ of $\cH_q$. For more details, the Reader is referred to  \cite{hoffer1972,huppertI1967}.
\begin{result}\cite[Theorem 12.24 (iv), Proposition 11.30]{HKT}
\label{sect12.3}  
$\aut(\mathbb{F}_{q^2}(\cH_q))$ $\cong \PGU(3,q)$ and $\aut(\mathbb{F}_{q^2}(\cH_q))$ acts on the set of all $\mathbb{F}_{q^2}$-rational points of $\cH_q$ as $\PGU(3,q)$ in its natural doubly transitive permutation  representation of degree $q^3+1$ on the isotropic points of the unitary polarity of the projective plane $\PG(2,\mathbb{F}_{q^2})$. Furthermore, 
$\aut(\mathbb{F}_{q^2}(\cH_q))$ $\cong$ $\aut(\mathbb{F}(\cH_q))$. 
\end{result}
The maximal subgroups of $\PGU(3,q)$ were determined by Mitchell \cite{mit} and Hartley \cite{har}, see also Hoffer \cite{hoffer1972}. Let $S_p$ be a Sylow $p$-subgroup of $\aut(\mathbb{F}_{q^2}(\cH_q))=\PGU(3,q)$. From Results \ref{lem15042023} and \ref{zeroprank}, it may be assumed up to conjugacy that the unique fixed point of $S_p$ is the point at infinity $Y_\infty$ of $\cH_q$.  The following result describes the structure of the stabiliser of $Y_\infty$ in $\aut(\mathbb{F}_{q^2}(\cH_q))$. 

\begin{result}\cite[Section 4]{GSX}
\label{struct} Let the Hermitian function field be given by its canonical form $\mathbb{F}_{q^2}(x,y)$ with $y^q+y-x^{q+1}=0$. Then the  stabiliser $G$ of $Y_\infty$ in $\aut(\mathbb{F}_{q^2}(\cH_q))$ consists of all maps
\begin{equation}
    \label{correzione}
    \varphi_{a,b,\lambda}:\,(x,y)\mapsto (\lambda x+a, a^q \lambda x +\lambda^{q+1}y +b)
 \end{equation}
 where 
 \begin{equation}
 \label{eqA250523}
 a,b,\lambda\in \mathbb{F}_{q^2},\,\,b^q+b=a^{q+1},\,\,\lambda\not=0. 
\end{equation}
In particular, $G$ contains a subgroup of order $q^3(q-1)$ consisting of the maps
\begin{equation} 
 \label{eq250523}
 \psi_{a,b,\lambda}:\,(x,y)\mapsto (\lambda x+a, a^q \lambda x +y +b)
 \end{equation}
 where 
 \begin{equation}
 \label{eqA250523}
 a,b,\lambda\in \mathbb{F}_{q^2},\,\,b^q+b=a^{q+1},\,\,\lambda^{q+1}=1. 
 \end{equation}
 Also, $G=S_p\rtimes C$ where $S_p=\{\psi_{a,b,1}| b^q+b=a^{q+1},\,\,a,b\in \mathbb{F}_{q^2}\}$ and $C=\{\psi_{0,0,\lambda}|\lambda^{q+1}=1,\,\,\lambda\in \mathbb{F}_{q^2}\}$.
 \end{result} From Result \ref{struct}, $S_p$ is the Sylow $p$-subgroup of the stabiliser of $Y_\infty$ in $\aut(\mathbb{F}_{q^2}(\cH_q))$.
\begin{result}
\label{conjcl} $S_p$ has the following properties.
\begin{itemize}
\item[(i)] The center $Z(S_p)$  of $S_p$ has order $q$ and it consists of all maps $\psi_{0,b,1}$ with $b^q+b=0,\,\, b\in\mathbb{F}_{q^2}$. Also, $Z(S_p)$ is an elementary abelian group of order $p$.
\item[(ii)]  The non-trivial elements of $S_p$ form two conjugacy classes in  the stabiliser of $Y_\infty$ in $\aut(\mathbb{F}_{q^2}(\cH_q))$, one comprises all non-trivial elements of $Z(S_p)$, the other does the remaining $q^3-q$ elements.
\item[(iii)] The elements of $G$ other than those in $Z(S_p)$ have order $p$, or $p^2=4$ according as $p>2$ or $p=2$.
\end{itemize}
\end{result}
As a corollary, the subgroups of order $p^2$ of $S_p$ are determined.
\begin{result}
\label{p2} Up to conjugacy in the stabiliser of $Y_\infty$ in $\aut(\mathbb{F}_{q^2}(\cH_q))$, the subgroups of $S_p$ of order $p^2$ in terms of their generators are the following.

For $p$ odd, every subgroup of order $p^2$ of $S_p$ is an elementary abelian group. There are two such families of subgroups of order $p^2$, named of type $\mathfrak{U}$ and $\mathfrak{V}$, according as it is contained in $Z(S_p)$ or not. More precisely, if 
$$
U_b=\langle \psi_{0,1,1},\psi_{0,b,1}\rangle;\, b\in \mathbb{F}_{q}\setminus\mathbb{F}_p,\qquad
V_c=\langle \psi_{1,\nicefrac{1}{2},1}, \psi_{0,c,1}\rangle;\, c^q+c=0, \,c\in\mathbb{F}_{q^2}.
$$
then $\mathfrak{U}=\{U_b|b\in \mathbb{F}_{q}\setminus\mathbb{F}_p\}$ and $\mathfrak{V}=\{V_c|c^q+c=0,\,c\in\mathbb{F}_{q^2}\}.$
For $p=2$, a subgroup of order $p^2=4$ of $S_p$ is either elementary abelian, or cyclic, according as it is contained in $Z(S_p)$ or not. In the former case, such a subgroup is generated by $\psi_{0,1,1}$ and $\psi_{0,b,1}$, that is as $U_b$ for $p$ odd, and it is named of type $\mathfrak{U}$; in the cyclic case, up to conjugacy, it is generated by
$\psi_{1,c,1}$ with $c^q+c+1=0, \,c\in \mathbb{F}_{q^2}$.
\end{result}
\begin{rem} 
\label{rem250523} If the Hermitian function field be given by the canonical form $\mathbb{F}_{q^2}(x,y)$ with $y^q+y+\omega x^{q+1}=0$ where $\go\in \mathbb{F}_{q^2}$ if a fixed element such that $\omega^{q-1}=-1$, then the above Results \ref{struct}, \ref{conjcl} and  \ref{p2} remain valid up to the following changes: In (\ref{eq250523}), $\psi_{a,b,\lambda}$ is replaced by
$$\varphi_{a,b,\lambda}:\,(x,y)\mapsto (\lambda x+a, a^q \lambda \omega x +y +b),$$
and (\ref{eqA250523}) is replaced by
$$ a,b,\lambda\in \mathbb{F}_{q^2},\,\,b^q+b+\omega a^{q+1},\,\lambda^{q+1}=1.$$ 
\end{rem} 

The Galois subcovers of $\mathbb{F}_{q^2}(\cH_q)$ with respect to a subgroup $H$ of prime order were thoroughly classified in \cite{cossidente-korchmaros-torres2000}. For the case $|H|=p$, the classification is reported in the following result.
\begin{result}\cite[Theorem 5.74]{HKT}
\label{ckt} Let $H$ be a subgroup of $\aut(\mathbb{F}_{q^2}(\cH_q))$ of order $p$. The Galois subcover of $\mathbb{F}_{q^2}(\cH_q)$ with respect to $H$ is $\mathbb{F}_{q^2}$-isomorphic to the function field $\mathbb{F}_{q^2}(\xi,\eta)$ where either \rm(i) or \rm(ii) hold:
\begin{itemize}
\item[(i)] $\sum_{i=1}^{h}\eta^{\nicefrac{q}{p^i}}+\go \xi^{q+1}=0$ with $\omega^{q-1}=-1$, and $H$ is in the center of a Sylow $p$-subgroup of $\aut(\mathbb{F}_{q^2}(\cH_q))$;
\item[(ii)] $\eta^q+\eta -(\sum_{i=1}^h\ \xi^{\nicefrac{q}{p^i}})^2=0$ for $p>2$, and $H$ is not in the center of a Sylow $p$-subgroup of $\aut(\mathbb{F}_{q^2}(\cH_q))$.
\end{itemize}
\end{result}
The following result is a corollary of \cite[Section 3]{GSX}. 
\begin{result} 
\label{gsx1}  Let $\mathfrak{g}$ be the genus of the Galois cover of $\mathbb{F}_{q^2}(\cH_q)$ with respect to a subgroup $H$ of $\aut(\mathbb{F}_{q^2}(\cH_q))$ of order $p^2$. If $S_p$ is the unique Sylow $p$-subgroup of  $\aut(\mathbb{F}_{q^2}(\cH))$ containing $H$, and  $h$ denotes the order of the subgroup of $H\cap Z(S_p)$ then one of the following cases occurs.    
$$\mathfrak{g}=
\begin{cases}
\ha \frac{q}{p}(\frac{q}{p}-1)\,\,for\,\,h=1;\\
\ha \frac{q}{p^2}(q-1)\,\,for\,\, h=2. 
\end{cases}
$$
\end{result}

\section{Galois subcovers of $\mathbb{F}_{q^2}(\cH)$ with respect to subgroup of type $\mathfrak{U}$} 
As in Remark \ref{rem250523}, we take $\mathbb{F}_{q^2}(\cH_q)$ in its canonical form $\mathbb{F}_{q^2}(x,y)$ with $y^q-y+\omega x^{q+1}=0$ and $\omega^{q-1}=-1$. The group $\Phi=\langle \varphi_{0,1,1}\rangle$ has order $p$, and it is contained in $Z(S_p)$. Let $\eta=y^p-y$ and $\xi=x$. Then $\varphi_{0,1,1}(\eta)=\varphi_{0,1,1}(y^p-y)=\varphi_{0,1,1}(y)^p-\varphi_{0,1,1}(y)=(y+1)^p-(y+1)=y^p-y=\eta$. Moreover, $y^q-y=Tr(\eta)$. 
Since $\varphi_{0,1,1}$ fixes $\xi$, this shows that the Galois subcover $\mathbb{F}_{q^2}(\cF')$ of $\mathbb{F}_{q^2}(\cH_q)$ with respect to $\Phi$ is as in (i) of Result \ref{ckt}. That equation is
\begin{equation}
\label{eqiia}
\sum_{i=1}^{h}\ \eta^{\nicefrac{q}{p^i}}+\go x^{q+1}=0.
\end{equation}
Fix $b\in \fq$ with $b^p\neq b$. Then $\varphi_{0,b,1}$ commutes with $\varphi_{0,1,1}$, and hence $\varphi_{0,b,1}$ induces an automorphism $\varphi$ of $\mathbb{F}_{q^2}(\cF')$. More precisely, a straightforward computation shows that $\varphi$ is the map $\varphi:(\xi,\eta)\mapsto (\xi,\eta+b^p-b)$.  Let $\Phi_b$ be the  $\fq$-automorphism group of $\mathbb{F}_{q^2}(\cF')$ generated by $\varphi$. Then 
the Galois subcover  of $\mathbb{F}_{q^2}(\cH_q)$ with respect to $\langle \varphi_{0,1,1},\varphi_{0,b,1}\rangle$ is the Galois subcover $G_b$ of $\Fs(\cF')$ with respect to $\Phi_b$. 
\begin{theorem}
\label{propiia}
The Galois subcover $G_b$ of $\Fs(\cF')$ with respect to $\Phi_b$ has the following properties:
\begin{itemize}
\item[(I)] $G_b=\Fs(\xi,\rho)$ with
\begin{equation}
\label{eq80323}
\sum_{i=1}^{h-1}(b-b^{p^i})\rho^{p^{i-1}}+\go \xi^{q+1}=0
\end{equation}
where $q=p^h$.
\item[(II)] The genus of $G_b$ equals $\ha q(p^{h-2}-1)=\ha q(\frac{q}{p^2}-1).$
\item[(III)] Let $P_\infty$ be the unique point of $G_b$ at infinity. The Weierstrass semigroup at the unique place centered at $P_\infty$ is generated by $p^{h-2}=\frac{q}{p^2}$ and $q+1$.
\item[(IV)] The $\mathbb{F}_{q^2}$-automorphism group of $G_b$ is the semidirect product of a normal subgroup of order $\frac{q}{p^2}$ by a (cyclic) complement of order $(q+1)(p-1)$.
\item[(V)] Let $\bar{b}\in \mathbb{F}_q\setminus \mathbb{F}_p$. Then $G_b$ and $G_{\bar{b}}$ are $\mathbb{F}_{q^2}$-isomorphic in exactly one of the following cases
\begin{itemize}
     \item[(i)] $b,\bar{b}\in \mathbb{F}_{p^2}$, or $b,\bar{b}\in \mathbb{F}_{p^3}$,
     \item[(ii)] $b,\bar{b}\not\in \mathbb{F}_{p^2}\cup \mathbb{F}_{p^3}$ and
\begin{equation}\label{eqA040423} \bar{b}= \frac{\alpha b +\beta}{\gamma b + \delta}
\end{equation}
with $\alpha,\beta,\gamma,\delta \in \mathbb{F}_p$ and $\alpha \delta - \beta \gamma \neq 0$.
\end{itemize}
\item[(VI)] Assume that $\mathbb{F}_q$ has a proper subfield $\mathbb{F}_{p^m}$ larger than $\mathbb{F}_p$. If $b\in\mathbb{F}_p$ and $\bar{b}\in\mathbb{F}_q\setminus \mathbb{F}_{p^m}$ then $G_b$ and $G_{\bar{b}}$ are not $\mathbb{F}_{q^2}$-isomorphic.
\end{itemize}
\end{theorem}
\begin{proof} (I) We show first that the fixed field $F$ of $\Phi_b$ is generated by $\xi$ together with
\begin{equation}
\label{eq1iia}
\zeta=\eta^p-(b^p-b)^{p-1}\eta.
\end{equation} Since $\varphi(\xi)=\xi$ and
$$\varphi(\zeta)=\varphi(\eta)^p-(b^p-b)^{p-1}\varphi(\eta)=\eta^p-(b^p-b)^p-(b^p-b)^{p-1}(\eta+b^p-b)=
\eta^p-(b^p-b)^{p-1}\eta=\zeta,$$
we have $\Fs(\xi,\zeta)\subseteq F$. Furthermore, $[\Fs(\cF'):\Fs(\xi,\zeta)]=p$. Since $p$ is prime, this yields either $\Fs(\xi,\zeta)=F$ or $F=\Fs(\cF')$. The latter case cannot actually occur, and hence $F=\Fs(\xi,\zeta)$. Therefore $F=G_b$.
We have to eliminate $\eta$ from equations (\ref{eqiia}) and (\ref{eq1iia}). Let
$\tau=(b^p-b)^{-1}\eta.$ Replacing $\eta$ by $(b^p-b)\tau$ in (\ref{eq1iia}) gives
\begin{equation}
\label{eq2iia} \zeta=(b^p-b)^p(\tau^p-\tau).
\end{equation}
Let $\rho=\tau^p-\tau$. Then $$\rho+\rho^p+\ldots+\rho^{p^{h-2}}=\tau^{p^{h-1}}-\tau.$$
Furthermore, from $\eta=(b^p-b)\tau$,
$$\eta+\eta^p+\ldots+\eta^{p^{h-1}}=b^q\tau^{p^{h-1}}-b\tau-\sum_{i=1}^{h-1} b^{p^i} (\tau^p-\tau)^{p^{i-1}}=b(\tau^{p^{h-1}}-\tau)-\sum_{i=1}^{h-1} b^{p^i}(\tau^p-\tau)^{p^{i-1}}.$$
Therefore
$$\go\xi^{q+1}+b(\rho+\rho^p+\ldots+\rho^{p^{h-2}})-\sum_{i=1}^{h-1} b^{p^i}\rho^{p^{i-1}}$$
whence (\ref{eqiia}) follows.

(II) The claim follows from \cite[Lemma 12.1(b)]{HKT}.

(III) By (\ref{eq80323}), \cite[Lemma 12.2]{HKT} applies for $n=h-2$ and $m=q+1$, and (III) follows from the remark after the proof of that lemma. 


(IV) For $a\in \mathbb{F}_{q^2}$ let $v\in \mathbb{F}_{q^2}$ be any element satisfying $v^q-v+\go a^{q+1}=0$. A straightforward computation shows that the map $\psi_{a,v,1}$
\begin{equation*}
\label{eq7mar2023}
\begin{cases} \xi'=\xi+a\\
              \rho':=\rho+\go a^{qp^2}\xi^{p^2}-(\go^p a^{p^2}+(b^p-b)^{p-1}\go^p a^{qp})\xi^p-(b^p-b)^{p-1}\go a^q \xi+
              (v^p-v)^p-(b^p-b)^{p-1}(v^p-v)
\end{cases}
\end{equation*}
is an $\mathbb{F}_{q^2}$-automorphism of $G_b$. Also, $\psi_{a,v,1}=1$ if and only if $a=0$ and
$(v^p-v)^p-(b^p-b)^{p-1}(v^p-v)=0$. Therefore, these maps form a $\mathbb{F}_{q^2}$-automorphism group $V$ of $G_b$ of order $\frac{q^3}{p^2}$. From \cite[Lemma 12.1(ii)(c)]{HKT} $P_\infty$ is the unique place of $G_b$ centered at a point at infinity. Thus $\psi_{a,v,1}$ fixes $P_\infty$.

For any $\mu\in \mathbb{F}_p^*$ let $\lambda\in \mathbb{F}_{q^2}$ be any element such that $\lambda^{q+1}=\mu$. Then the map $\tau_{\lambda,\mu}$
\begin{equation}
\label{eqA7032023}
\begin{cases}
\xi'=\lambda \xi\\
\rho'=\mu \rho
\end{cases}
\end{equation}
is an $\mathbb{F}_{q^2}$-automorphism of $G_b$. These maps form an $\mathbb{F}_{q^2}$-automorphism group $\Lambda$ of $G_b$ of order $(q+1)(p-1)$. Furthermore, $\Lambda$ normalizes $V$ and hence they generate an $\mathbb{F}_{q^2}$- automorphism group $W=V\rtimes \Lambda$ of $G_b$ of order $(q+1)(p-1)\frac{q^3}{p^2}$.
We show that $W$ is the full $\mathbb{F}_{q^2}$-automorphism group of $G_b$.

Since $G_b$ is a $\mathbb{F}_{q^2}$-maximal function field of genus $\mathfrak{g}=\ha q(\frac{q}{p^2}-1)$ it has as many as $1+\frac{q^3}{p^2}$ $\mathbb{F}_{q^2}$-rational places. Assume that $\aut(G_b)$ has a $p$-subgroup $Z$ properly containing $V$. By Result \ref{zeroprank} $G_b$ has zero Hasse-Witt invariant. Therefore,  Result \ref{resth11.129} shows that each element in $Z$ has a exactly one fixed point. Therefore, $Z$ fixes the place centered at $P_\infty$ and acts semi-regularly on the set of the remaining $\frac{q^3}{p^2}$ $\mathbb{F}_{q^2}$-rational places of $G_b$. From Result \ref{lem15042023}, $Z$ has order at most $\frac{q^3}{p^2}$ whence $Z=V$ follows.

From \cite[Theorem 12.11]{HKT}, $\aut(G_b)$ fixes  $P_\infty$. Let $\alpha\in\aut(G_b)$ be a generator of a Sylow $u$-subgroup $S_u$ of $\aut(G_b)$ with $u\ne p$. Since $G_b$ has as many as $\mathbb{F}_{q^2}$-rational places other than that centered at $P_\infty$, it turns out that $\alpha$ fixes an $\mathbb{F}_{q^2}$-rational place centered at an affine point. As $V$ is transitive on these places, we may assume up to conjugacy under an element $\beta\in S_p$ that $\alpha$ fixes the place $O$. Furthermore,  $P_\infty$ is the unique pole of $\xi$ and it has multiplicity $\frac{q}{p^2}$. From \cite[Lemma 11.13]{HKT}, $\alpha(\xi)=w\xi+v$. Here $v=0$ as $\alpha$ fixes the place $O$. 
From (III), $P_\infty$ the pole numbers smaller than equal to $q+1$ are $p^{h-2},2p^{h-2},\ldots,p^2(p^{h-2})=q,\,p^h+1=q+1$, a base of the Riemann-Roch space $\mathcal{L}((q+1)P_\infty)$ consists of  $1,\xi,\xi^2,\ldots \xi^{p^2},\rho$. Therefore $$\alpha(\rho)=c_1\xi+\ldots+c_{p^2}\xi^{p^2}+c\rho.$$

The zeros of $\xi$ are the places centered at the affine points $(0,t_j)$ of $G_b$ with $j=1,\ldots,p^{h-2}$ such that  $\sum_{i=1}^{h-1}(b-b^{p^i})t_j^{p^{i-1}}=0$. From the classical Vi\'ete formula,

$$t_1+t_2+\ldots t_{p^{h-2}}=-\frac{b-b^{p^{h-1}}}{b-b^{p^{h-2}}}.$$

Since $\alpha(\xi)=w\xi$, $\alpha$ takes zeros of $\xi$ to zeros of $\xi$, it follows that $\sum_{i=1}^{h-1}(b-b^{p^i})(ct_j)^{p^{i-1}}=0$ for $j=1,\ldots,p^{h-2}$. Therefore,

$$t_1+t_2+\ldots t_{p^{h-2}}=-\frac{b-b^{p^{h-1}}}{b-b^{p^{h-2}}}\cdot\frac{c^{p^{h-2}}}{c^{p^{h-3}}}.$$

\noindent Comparison shows that $c^{p-1}=1$, that is, $c\in \mathbb{F}_{p}^*$.

From (\ref{eq80323}),
$$\sum_{i=1}^{h-1}(b-b^{p^i})\alpha(\rho)^{p^{i-1}}+\go \alpha(\xi)^{q+1}=0,$$
whence
$$\sum_{i=1}^{h-1}(b-b^{p^i})(c_1\xi+\ldots+c_{p^2}\xi^{p^2})^{p^{i-1}}+\sum_{i=1}^{h-1}(b-b^{p^i})(c\rho)^{p^{i-1}}+\go w^{q+1}\xi^{q+1}=0.$$
Thus
$$\sum_{i=1}^{h-1}(b-b^{p^i})(c_1\xi+\ldots+c_{p^2}\xi^{p^2})^{p^{i-1}}+c\sum_{i=1}^{h-1}(b-b^{p^i})(\rho)^{p^{i-1}}+\go w^{q+1}\xi^{q+1}=0.$$

Now, from (\ref{eq80323}),
$$\sum_{i=1}^{h-1}(b-b^{p^i})(c_1\xi+\ldots+c_{p^2}\xi^{p^2})^{p^{i-1}}+\go (w^{q+1}-c)\xi^{q+1}=0.$$
Since $\xi$ is a transcendent element of $G_b$, this yields that the polynomial
$$G(Z)=\sum_{i=1}^{h-1}(b-b^{p^i})(c_1Z+\ldots+c_{p^2}Z^{p^2})^{p^{i-1}}+\go (w^{q+1}-c)Z^{q+1}$$ is the zero polynomial. Therefore, $w^{q+1}=c$, that is,
\begin{equation}
\label{eqB7032023}
\begin{cases}
\xi'=w \xi\\
\rho'=c \rho
\end{cases}
\end{equation}
which together with (\ref{eq80323}) yield $\alpha\in \Lambda$. Therefore, a conjugate of $S_u$ under an element $\beta\in S_p$ is in $\Lambda$. It turns out that $S_u$ is a subgroup of $S_p\rtimes \Lambda$. Thus claim (IV) follows as $G_b=S_p\rtimes C$ with a cyclic complement $C$ and any cyclic group is the direct product of its Sylow subgroups.

(V) Let $G_{\bar{b}}=\mathbb{F}_{q^2}(\bar{\xi},\bar{\rho})$ with
\begin{equation}
\label{eqA80323}
\sum_{i=1}^{h-1}(\bar{b}-\bar{b}^{p^i})\bar{\rho}^{p^{i-1}}+\go \bar{\xi}^{q+1}=0
\end{equation}
where $q=p^h$. Let $\bar{P}_\infty$ and $\bar{O}$ denote its places centered at the unique point at infinity and at the origin, respectively.

Assume that $G_b$ and $G_{\bar{b}}$ are $\mathbb{F}_{q^2}$-isomorphic. Let $\iota$ be an associated birational map $G_b\mapsto G_{\bar{b}}$ defined over $\mathbb{F}_{q^2}$. Then $\iota$ takes $P_\infty$ to $\bar{P}_\infty$ as the unique fixed place of $\aut(G_b)$ (resp. $\aut(G_{\bar{b}})$) are $P_\infty$ (resp. $\bar{P}_\infty$). Furthermore,
$\iota$ takes $O$ to an $\mathbb{F}_{q^2}$-rational place $\bar{V}$ of $G_{\bar{b}}$. If $\bar{\alpha}\in\aut(G_{\bar{b}})$ takes $\bar{V}$ to $\bar{O}$ then replacing $\iota$ by $\bar{\alpha}\circ\iota$ gives a birational map which takes $O$ to $\bar{O}$. Since $\div(\rho)=(q+1)O-(q+1)P_\infty$ we also have
$\div(\iota(\rho))=(q+1)\bar{O}-(q+1)\bar{P}_\infty$. Therefore, $\iota(\rho)/\bar{\rho}$ has neither zeros nor  poles. Thus $\iota(\rho)=c\bar{\rho}$ with $c\in \mathbb{F}_{q^2}^*$. Furthermore,  both $\bar{\xi}$ and $\iota(\xi)$ have the smallest pole number at $P_\infty$ and $\bar{P}_\infty$, respectively. Therefore $\iota(\xi)=\sigma \bar{\xi}$ for some $\sigma \in \mathbb{F}_{q^2}^*$. Let $\delta=\sigma^{q+1}$. From  (\ref{eq80323}),
\begin{equation}
\label{eqA100323}
\sum_{i=1}^{h-1}(b-b^{p^i})c^{p^{i-1}}\rho^{p^{i-1}}+\go \delta \xi^{q+1}=0.
\end{equation}
Comparison with (\ref{eqA80323}) shows that
\begin{equation}
\label{eqB100323} \delta(\bar{b}-\bar{b}^{p^i})=c^{p^{i-1}}(b-b^{p^i}) {\mbox{\,\,for $i=1,\ldots,h-1$.}}
\end{equation}
Since $b,\bar{b},\delta\in \mathbb{F}_q^*$, (\ref{eqB100323}) for $i=1$ yields $c\in \mathbb{F}_q^*$.
Conversely, if there exist $c,\delta\in \mathbb{F}_q^*$ such that (\ref{eqB100323}) holds then the rational map $\iota: G_b \mapsto G_{\bar{b}}$ where $\iota(\rho)=c\bar{\rho}$ and $\iota(\xi)=\sigma\bar{\xi}$  with $\sigma^{q+1}=\delta$ is an $\mathbb{F}_{q^2}$-isomorphism.
Three cases are treated separately.
\subsubsection{Case $b^{p^2}-b=0$} For $h=2$, $b$ and $\bar{b}$ are linearly independent over $\mathbb{F}_p$ so that $\bar{b}=\alpha b+ \beta$ with $\alpha,\beta\in \mathbb{F}_p$ and $\alpha\neq 0$. Then (\ref{eqB100323}) holds for $i=1=h-1$, and thus $G_b$ and $G_{\bar{b}}$ are $\mathbb{F}_{q^2}$-isomorphic. Assume $h\ge 3$. If $G_b$ and $G_{\bar{b}}$ are $\mathbb{F}_{q^2}$-isomorphic then (\ref{eqB100323}) for $i=2$ yields $\bar{b}^{p^2}-\bar{b}=0$. Then both $b$ and $\bar{b}$ are elements of the (unique) subfield $\mathbb{F}_{p^2}$ of $\mathbb{F}_q$. Conversely, if $b,\bar{b}\in \mathbb{F}_{p^2}$ then they are linearly independent over $\mathbb{F}_p$ and hence
$\bar{b}=\alpha b+ \beta$ with $\alpha,\beta\in \mathbb{F}_p$ and $\alpha\neq 0$. Therefore, (\ref{eqB100323}) holds for $=1,2,\ldots h-1$, and hence $G_b$ and $G_{\bar{b}}$ are $\mathbb{F}_{q^2}$-isomorphic.

\subsubsection{Case $b^{p^3}-b=0$}  For $h=3$, $b,\bar{b}\in \mathbb{F}_{p^3}\setminus \mathbb{F}_{p}$ and hence (\ref{eqA040423}) holds. A straightforward computation shows that then (\ref{eqB100323}) holds for $i=1,2=h-1$. Therefore, $G_b$ and $G_{\bar{b}}$ are $\mathbb{F}_{q^2}$-isomorphic. Assume $h\ge 4$. For $i=1,2$, (\ref{eqB100323}) reads
\begin{equation}
\label{eqA230323J} \delta(\bar{b}-\bar{b}^{p})=c(b-b^{p}).
\end{equation}
and
\begin{equation}
\label{eqB230323} \delta(\bar{b}-\bar{b}^{p^2})=c^p(b-b^{p^2}),
\end{equation}
respectively. Since $\bar{b}^{p^3}-\bar{b}=0$ also holds, comparison of (\ref{eqA230323J}) with the $p$-power of (\ref{eqB230323}) yields
\begin{equation}
    \label{E230323}
    \delta^{p-1}=c^{p^2-1}.
\end{equation}
Conversely, take  $c\in \mathbb{F}_q$ so that \begin{equation}
    \label{eqA290323K}
    c^p=\frac{b-b^p}{\bar{b}-\bar{b}^p}.
\end{equation}
and define $\delta=c^{p+1}$. Then (\ref{E230323}) holds whence (\ref{eqB230323}) follows. Therefore, $G_b$ and $G_{\bar{b}}$ are $\mathbb{F}_{q^2}$-isomorphic.

\subsubsection{Case $b^{p^i}-b=0$ for some $4\le i\le h-1$ but $b^{p^i}-b \neq 0$ for $1\le i\le 3$.}
Assume first (\ref{eqA040423}) holds. Then
\begin{equation}
\label{eqB0030423}
\gamma b\bar{b}+\alpha b+ \delta \bar{b}+\beta=0,\,\alpha,\beta,\gamma \in \mathbb{F}_p.
\end{equation}
Since $G_b$ and $G_{\epsilon b+ \chi}$ are $\mathbb{F}_{q^2}$-isomorphic for $\epsilon,\chi \in \mathbb{F}_p$ with $\epsilon \neq 0$,  we may replace $b$ with $\epsilon b+\chi$ and $\bar{b}$ with $\bar{\epsilon}\bar{b}+\bar{\chi}$ where $\epsilon\bar{\epsilon}=(\alpha \delta -\beta\gamma)$. Then (\ref{eqB0030423}) reduces to $b\bar{b}=1$. Take $b^{-p}$ for $c$, and define $\delta$ to be $-b$. Then (\ref{eqB230323}) holds for $i=1,2,\ldots, h-1$. Therefore, with these choices of $c$ and $\delta$, the map $\iota$ is a an $\mathbb{F}_q$-isomorphism from $G_b$ to $G_{\bar{b}}$ with $\bar{b}=b^{-1}$.

Assume that $G_b$ and $G_{\bar{b}}$ are $\mathbb{F}_{q^2}$-isomorphic. Comparison of (\ref{eqB100323}) for $i=1$ with (\ref{eqB100323})  for $i=2$ gives
   \begin{equation}
       \label{para290323}
       (\bar{b}-\bar{b}^{p^2})(b-b^p)=c^{p-1}(b-b^{p^2})(\bar{b}-\bar{b}^p).
   \end{equation}
 Similarly, comparison of (\ref{eqB100323}) for $i=1$ with (\ref{eqB100323}) for $i=3$ gives
 \begin{equation}
       \label{parb290323}
       (\bar{b}-\bar{b}^{p^3})(b-b^p)=c^{p^2-1}(b-b^{p^3})(\bar{b}-\bar{b}^p).
\end{equation}
 Moreover, comparison of the $(p+1)$-power of (\ref{para290323}) with (\ref{parb290323}) shows
 \begin{equation}
     \label{eqcar}
     (\bar{b}-\bar{b}^{p^3})(\bar{b}-\bar{b}^p)^p(b-b^{p^2})^{p+1}=(b-b^{p^3})(b-b^p)^p(\bar{b}-\bar{b}^{p^2})^{p+1}.
 \end{equation}
 Therefore (\ref{eqA040423}) follows from Lemma \ref{lem060423}.

(VI) The claim is a corollary to (V).
\end{proof}

\section{Galois subcovers of $\mathbb{F}_{q^2}(\cH)$ with respect to subgroup of type $\mathfrak{V}$}
This time, take $\mathbb{F}_{q^2}(\cH_q)$ in its canonical form $\mathbb{F}_{q^2}(x,y)$ with $y^q+y-x^{q+1}=0$. 
The group $\Phi=\langle \psi_{1,\nicefrac{1}{2},1}\rangle$ has order $p$, and it is not contained in $Z(S_p)$. Let $\xi=x^p-x$ and $\eta=y-\ha x^2$. A straightforward computation shows that $\psi_{1,\nicefrac{1}{2},1}(\xi)=\xi$ and $\psi_{1,\nicefrac{1}{2},1}(\eta)=\eta$. Moreover, 
$$y^q+y-x^{q+1}=\eta^q+\ha x^{2q}+\eta+\ha x^2-x^{q+1}=\eta^q+\eta+\ha x^{2q}+\ha x^2-x^{q+1}=\eta^q+\eta+\ha (x^q-x)^2.$$
Since $Tr(\xi)=x^q-x$, this gives
$$\eta^q+\eta+\ha(x^q-x)^2=\eta^q+\eta+\ha Tr(\xi)^2.$$
Therefore, the Galois subcover $\mathbb{F}_{q^2}(\cF')$ of $\mathbb{F}_{q^2}(\cH_q)$ with respect to $\Phi$ is 
$\mathbb{F}_{q^2}(\xi,\eta)$ with 
\begin{equation}
\label{eqiib}
\eta^q+\eta+\ha \Big(\sum_{i=1}^{h}\xi^{p^{i-1}}\Big)^2=0.
\end{equation}
In particular, this shows that $\mathbb{F}_{q^2}(\cF')$ is $\mathbb{F}_{q^2}$-isomorphic to (ii) of Result \ref{ckt}.

Fix $b$ such that $b^q+b=0$. Then $\psi_{0,b,1}$ commutes with $\psi_{0,\nicefrac{1}{2},1}$, and hence $\psi_{0,b,1}$ induces an $\mathbb{F}_q$-automorphism $\varphi$ of $\Fs(\cF')$. A straightforward computation shows that the map $\varphi:(\xi,\eta)\mapsto (\xi,\eta+b)$ is this $\fq$-automorphism of $\Fs(\cF')$. Let $\Phi_{b}$ be the  $\fq$-automorphism group of $\cF'$ generated by $\varphi$. Then the Galois subcover of $\Fs(\cH_q)$ with respect to $\langle \psi_{1,\nicefrac{1}{2},1}, \psi_{0,b,1}\rangle$ is the Galois subcover $G_b$ of $\Fs(\cF')$ with respect to $\Phi_b$.
 \begin{theorem}
\label{propiib} Let $b^q+b=0$. Then the Galois subcover $G_b$ of $\Fs(\cF')$ with respect to $\Phi_{b}$ has the following properties:
\begin{itemize}
\item[(I)] $G_b=\Fs(\xi,\rho)$ with
\begin{equation}
\label{eq80415}
\Big(\sum_{i=1}^{h}\xi^{p^{i-1}}\Big)^2-2b\sum_{i=1}^{h}\rho^{p^{i-1}}=0
\end{equation}
where $q=p^h$.
 \item[(II)] The genus of $G_b$ equals $\frac{1}{2}p^{h-1}(p^{h-1}-1)=\frac{1}{2}\frac{q}{p}(\frac{q}{p}-1)$.
\item[(III)] Let $P_\infty$ be the unique point of $G_b$ at infinity. The Weierstrass semigroup at the unique place centered at $P_\infty$ is generated by $\frac{q}{p}, \frac{q}{p}+\frac{q}{p^2}$ and $q+1$.
\item[(IV)] The $\mathbb{F}_{q^2}$-automorphism group of $G_b$ is the semidirect product of a normal subgroup of order $\frac{q^2}{p}$ by a (cyclic) complement of order $p-1$.
\item[(V)] Let $\bar{b}\in \mathbb{F}_q\setminus \mathbb{F}_p$. Then $G_b$ and $G_{\bar{b}}$ are $\mathbb{F}_{q^2}$-isomorphic if and only if there exists $\kappa\in \mathbb{F}_{p}^*$ such that $\bar{b}=\kappa b$.
 \end{itemize}
\end{theorem}
\begin{proof}
(I) We show first that the fixed field $F$ of $\Phi_{\nicefrac{1}{2}}$ is generated by $\xi$ together with
\begin{equation}
\label{eq1iib}
\zeta=\eta^p-b^{p-1}\eta.
\end{equation}
Since $\varphi(\xi)=\xi$ and
$$\varphi(\zeta)=\varphi(\eta)^p-b^{p-1}\varphi(\eta)=
\eta^p-b^{p-1}\eta=\zeta,$$
we have $\Fs(\xi,\zeta)\subseteq F$. Furthermore, $[\Fs(\cF'):\Fs(\xi,\zeta)]=p$. Since $p$ is prime, this yields either $\Fs(\xi,\zeta)=F$ or $F=\Fs(\cF')$. The latter case cannot actually occur, and hence $F=\Fs(\xi,\zeta)$. Therefore $F=G_b$.
We have to eliminate $\eta$ from equations (\ref{eqiib}) and (\ref{eq1iib}). Let
$\tau=b^{-1}\eta.$ Replacing $\eta$ by $b\tau$ in (\ref{eq1iib}) gives
\begin{equation}
\label{eq2iib} \frac{\zeta}{b^p}=\tau^p-\tau.
\end{equation}
Let $\rho=\tau^p-\tau$. Then $$\rho+\rho^p+\ldots+\rho^{p^{h-1}}=\tau^{p^{h}}-\tau.$$
Furthermore, from $\eta=b\tau$,
$$\eta^{p^{h}}+\eta=b^{p^{h}}\tau^{p^{h}}+b\tau=-b\tau^{p^{h}}+b\tau=-b(\tau^{p^{h}}-\tau).$$
Therefore
$$\eta^{p^{h}}+\eta=-b(\rho+\rho^p+\ldots+\rho^{p^{h-1}})$$
whence (\ref{eqiib}) follows.

(II) The claim follows from \cite{GSX}.


(III) We go on with some preliminary results on $\aut(G_b)$. Let $a,c\in \mathbb{F}_{q^2}$ and $\nu_a\in \mathbb{F}_p$.  A straightforward computation shows that the map
$\psi_{a,c}$
\begin{equation}
\label{eq15apr2023}
\begin{cases} \xi':=\xi+a\\
              \rho':=\rho+\nu_a \xi+c
\end{cases}
\end{equation}
is an $\mathbb{F}_{q^2}$-automorphism of $G_b$ if and only if $Tr(a)=\nu_a b$ and $Tr(a)^2 - 2bTr(c)=0$.  We show that there exist as many as $\frac{q^2}{p}$ such maps $\psi_{a,c}$. For any element $e\in \mathbb{F}_{q^2}$, consider its trace $Tr(e)$ over $\mathbb{F}_p$, i.e.  $Tr(e)=e+e^p+\ldots+e^{q/p}+e^q+\ldots +e^{\nicefrac{q^2}{p}}$. Observe that $T(e)=Tr(e)^q+Tr(e)$. Let $E$ be the additive subgroup of $\mathbb{F}_{q^2}$ consisting of all elements $e$ with $Tr(e)=0$. Clearly, $E$ has order $\frac{q^2}{p}$. The map $f: E\mapsto \mathbb{F}_{q^2}$ taking $e\in E$ to $Tr(e)$ is additive and its value set is contained in the additive subgroup of all elements $u$ such that $u^q+u=0$. More precisely, $f$ is surjective, as its nucleus comprises $\frac{q}{p}$  elements $e$ with $Tr(e)=0$. Since $b^q+b=0$, this yields that Equation $Tr(e)=b$ has exactly $\frac{q}{p}$  solutions.
Thus there exist as many as $q$ pairs $(a,\nu_a)$ with $\nu_a\in \mathbb{F}_p^*$ such that $Tr(a)=\nu_a b$. For each such pair $(a,\nu_a)$ and for each $\delta\in \mathbb{F}_{q^2}^*$ with $Tr(\delta)=0$, let $c=\ha (\nu_a a-\delta)$. Then $Tr(c)=\ha \nu_a Tr(a)$ whence $Tr(a)^2-2bTr(c)=0$. Since $Tr(\delta)=0$ has $\frac{q}{p}$  solutions, the claim follows. Thus the above maps $\psi_{a,c}$ form an elementary abelian subgroup $\Psi$ of $\aut(G_b)$ of order $\frac{q^2}{p}$. Actually, by  Lemma \ref{lem15042023}, $\Psi$ is the (unique) Sylow $p$-subgroup of $\aut(\cF')$. 
The commutator subgroup $\Psi'$ of $\Psi$ coincides with $\Gamma$.
In fact, any commutator $\psi_{a,c}^{-1}\psi_{\bar{a},\bar{c}}^{-1}\psi_{a,c}\psi_{\bar{a},\bar{c}}$ fixes $\xi$ and hence it belongs to $\Gamma$, and any element of $\Gamma$ can be obtained in that way.
Two elementary abelian subgroups of $\Psi$ play a role in the structure of $\Psi$, namely $\Gamma=\{\psi_{0,c}| Tr(c)=0\}$ and $\Delta=\{\psi_{a,0}|Tr(a)=0\}$. Here, $|\Gamma|=|\Delta|=\frac{q}{p}$, and  
$\Omega=\Gamma\Delta=\Gamma\times \Delta$ consisting of all $\psi_{a,c}$ with $\nu_a=0$. Also,   $\psi_{a,c}$ and $\psi_{\bar{a},\bar{c}}$ commute if and only if $\nu_a \bar{a}=\nu_{\bar{a}}a$. The latter condition occurs when either $a=\nu_a=0$, or $\bar{a}=\nu_{\bar{a}}=0$, or $\nu_a=\nu_{\bar{a}}=0$, or  
$\bar{a}=\nu a$ for some $\nu\in \mathbb{F}_p^*$. In particular, $Z(\Phi)=\Gamma$. It turns out that
the centralizer $C_{\Psi}(\psi_{a,c})$ has order either $\frac{q^2}{p}$, or $(\frac{q}{p})^2$, or $q$ according as $\psi_{a,c}\in \Gamma$, or $\psi_{a,c}\in \Omega\setminus\Delta$, or $\psi_{a,c}\in \Psi\setminus \Omega$. From this, each group-automorphism of $\Psi$ leaves $\Omega$ invariant, i.e. $\Omega$ is a characteristic subgroup of $\Psi$. Moreover, the commutators of $\Psi$ are exactly the elements of $\Gamma$, and hence $\Gamma$ coincides with the commutator subgroup $\Psi'$ of $\Psi$.

Let $\zeta\in \mathbb{Z}$ be the smallest non-gap of $G_b$ at $P_\infty$ and take $m\in G_b$ such that $\div(m)_\infty =\zeta P_\infty$. Since $P_\infty$ is fixed by $\Psi$, $\mathbb{K}(m)$ is left invariant by $\Psi$. Thus,  $\Psi$ induces an $\mathbb{F}_{q^2}$-automorphism group on $\mathbb{F}_{q^2}(m)$. Let $\Xi$ be the subgroup of $\Psi$ fixing $m$. Since $\Psi$ is a $p$-group, $\Psi/\Xi$ is an elementary abelian group, 
by Dickson's classification of subgroups of $\PGL(2,q)$. Then $\Xi$ contains the commutator subgroup $\Psi'$ of $\Psi$; see for instance From  \cite[Section I, Satz 8.2]{huppertI1967}. As we have already pointed out, $\Psi'=\Gamma$. Thus $\Gamma \le \Xi$.
Since the fixed subfield of $\Gamma$ in $G_b$ is $\mathbb{F}_{q^2}(\xi)$, this yields that $\mathbb{F}_{q^2}(\xi)$ contains $\mathbb{F}_{q^2}(m)$. Then $[G_b:\mathbb{F}_{q^2}(m)]\geq [G_b:\mathbb{F}_{q^2}(\xi)]$ whence $\div(\xi)_\infty=\zeta P_\infty$. Thus $\xi$ may be chosen for $m$, and the smallest non-gap at $P_\infty$ is equal to $\frac{q}{p}$.
We show that $\frac{q}{p}+\frac{q}{p^2}$ is another non-gap of $G_b$ at $P_\infty$. More precisely, we show that $\tau=\rho-\kappa\xi^2$ for $\kappa^q=\ha \frac{1}{b^p}\in \mathbb{F}_{q^2}$ has a (unique) pole at $P_\infty$ of multiplicity $\frac{q}{p}+\frac{q}{p^2}$. Replace $\rho$ with $\tau+\kappa \xi^2$ in (\ref{eq80415}).
From the choice of $k$ there exists a polynomial $h(X)\in \mathbb{F}_{q^2}[X]$ of degree $\frac{q}{p}+\frac{q}{p^2}$ such that
\begin{equation}
\label{eq200423}
h(\xi)-2b\sum_{i=1}^{h}\tau^{p^{i-1}}=0.
\end{equation}
Let $\cC$ be the plane curve associated to $G_b$ i.e. $\cC$ has equation
$$\Big(\sum_{i=1}^{h}X^{p^{i-1}}\Big)^2-2b\sum_{i=1}^{h}Y^{p^{i-1}}=0.$$
Let $\cD$ be the plane curve of equation
$$H(X,Y)=h(X)-2b\sum_{i=1}^{h}Y^{p^{i-1}}\in \mathbb{F}_{q^2}[X,Y].$$
Clearly $H(\xi,\tau)=0$. Moreover, $\cD$ is the birational image of $\cC$ by the map $(X,Y)\mapsto (X,Y-k X^2)$ which fixes the place $P_\infty$.  Since $\deg_X(H(X,Y))>\deg_Y(H(X,Y))$, $\cD$ has a unique place centered at a point at infinity, namely $P_\infty$. Therefore, $P_\infty$ is the unique pole of $\tau$, and the relative pole number equals $\deg_X(H(X,Y))=\frac{q}{p}+\frac{q}{p^2}$.
Thus $\frac{q}{p}+\frac{q^2}{p}$  is a non-gap of $G_b$ at $P_\infty$. Therefore, $\frac{q}{p},\frac{q}{p}+\frac{q^2}{p}$ and $q+1$ are non-gaps of $G_b$ at $P_\infty$. 

Let $d_0=0,\, d_1=\frac{q}{p},\, d_2=\frac{q}{p^2},\, d_3=1$, and $A_1=\{1\},\, A_2=\{1,p\},$ $A_3=\{\frac{q}{p},\frac{q}{p}+\frac{q}{p^2},q+1\}$. Then  $1\in S_1, p+1\in S_2$ and $q+1\in S_3$. Thus 
the sequence  $\{\frac{q}{p},\frac{q}{p}+\frac{q}{p^2},q+1\}$ is telescopic. 
 Furthermore, $$l_g(S_3)=-\frac{q}{p}+(p-1)\Big(\frac{q}{p}+\frac{q}{p^2}\Big)+\Big(\frac{q}{p^2}-1\Big)(q+1)=\frac{q^2}{p^2}-\frac{q}{p}-1.$$  Now, Result \ref{resgeneresemigroup} yields (III)



(IV) 
By \cite[Theorem 11.49 (ii)(b)]{HKT}, $\aut(G_b)=\Psi\rtimes \Lambda$ with a cyclic subgroup $\Lambda$. For $\lambda \in \mathbb{F}_p^*$, the map $\tau_\lambda$
\begin{equation}
\label{eq17042023}
\begin{cases}
\xi':=\lambda \xi\\
\rho':=\lambda^2\rho
\end{cases}
\end{equation}
is an $\mathbb{F}_{q^2}$-automorphism of $G_b$. 
We show that $|\Lambda|=p-1$, that is, $\Lambda=\{\tau_\lambda|\lambda\in \mathbb{F}_p^*\}$. Let $\bar{G}_b$ the fixed subfield  of $\Omega$ in $G_b$. From $|\Omega|=(\frac{q}{p})^2>(\frac{q}{p})(\frac{q}{p}-1)=2\mathfrak{g}(G_b)$, $\bar{G}_b$ is rational; see \cite[Theorem 11.78]{HKT}. Since $\Omega$ is a characteristic subgroup of $\Psi$, the quotient group $\bar{\Omega}=\aut(G_b)/\Omega$ is a subgroup of $\aut(\bar{G}_b)$. 

From \cite[Theorem 11.49 (b)]{HKT}, $\bar{\Omega}\cong \bar{U}\rtimes \Lambda$ where $\bar{U}$ is a $p$-group. 
In particular $\bar{U}\le \Psi/\Omega$ as $\Psi$ is the unique Sylow $p$-subgroup of $\aut(G_b)$. As we have already pointed out
$[\Psi:\Omega]=p$, this yields $|\bar{U}|=p$. 
Furthermore,
$\bar{\Omega}$ is isomorphic to a (solvable) subgroup of $\PGL(2,\mathbb{F}_{q^2})$. From Dickson's classification of subgroups of $\PGL(2,q)$, see \cite{maddenevalentini1982} and \cite[Theorem A.8]{HKT}, $\bar{\Omega}$ has order at most $p(p-1)$.  Therefore, $|\Lambda|=p-1.$ 

From the above arguments, we may also deduce that the action of $\aut(G_b)$ on the set $G_b(\mathbb{F}_{q^2})$ of $\mathbb{F}_{q^2}$-rational points has one fixed point, namely $P_\infty$, another short orbit of length $\frac{q^2}{p}$, namely that of $O$, and $1+(\frac{q}{p}-1)/(p-1)$ orbits of maximum length $(p-1)\frac{q^2}{p}$. In fact, $\aut(G_b)$ may have only two short orbits as the subcover $G_b/\Psi$ is rational.

(V) Take another $\bar{b}\in \mathbb{F}_{q^2}$ with $\bar{b}^q+\bar{b}=0$ and let $G_{\bar{b}}=\mathbb{F}_{q^2}(\bar{\xi},\bar{\rho})$ with $Tr(\bar{\xi})^2-2\bar{b}Tr(\bar{\rho})=0.$
Let $\bar{P}_\infty$ and $\bar{O}$ denote its places centered at the unique point at infinity and at the origin, respectively.

Assume that $G_b$ and $G_{\bar{b}}$ are $\mathbb{F}_{q^2}$-isomorphic. Let $\iota$ be an associated birational map $G_b\mapsto G_{\bar{b}}$ defined over $\mathbb{F}_{q^2}$. 

Then $\aut(G_{\bar{b}})=\iota\circ \aut(G_b)\circ\iota^{-1}$. Moreover, $\iota$ takes $P_\infty$ to $\bar{P}_\infty$ as the unique fixed place of $\aut(G_b)$ (resp. $\aut(G_{\bar{b}})$) are $P_\infty$ (resp. $\bar{P}_\infty$). Furthermore, $\iota$ takes $O$ to an $\mathbb{F}_{q^2}$-rational place $\bar{V}$ of $G_{\bar{b}}$. Since $\Lambda$ fixes $O$, $\iota\circ\Lambda\circ\iota^{-1}$ is a subgroup of $\aut(G_{\bar{b}})$ which fixes $\bar{V}$. Therefore, $\bar{V}$ is in the short orbit of $\aut(G_{\bar{b}})$, in particular in the same $\aut(G_{\bar{b}})$-orbit of $\bar{O}$. Thus  
there exists $\bar{\alpha}\in\aut(G_{\bar{b}})$ which takes $\bar{V}$ to $\bar{O}$.  Replacing $\iota$ by $\bar{\alpha}\circ\iota$ gives a birational map which takes $O$ to $\bar{O}$.

Both $\bar{\xi}$ and $\iota(\xi)$ have the same smallest pole number at $P_{\infty}$. Therefore $\iota(\xi)=\kappa\bar{\xi}$ for some $\kappa \in \mathbb{F}_{q^2}^*$.
Actually $\kappa\in \mathbb{F}_p^*$. In fact, $\bar{\xi}$ and $\iota(\xi)$ have the same zeros, namely the places centered at the $\frac{q}{p}$ (non-singular) points $P=(0,t)$ of $\Bar{\cC}$, and hence $Tr(t)=0$ and $Tr(kt)=0$ must hold simultaneously. This yields $\kappa\in \mathbb{F}_p^*$. It also implies the existence of $\bar{\alpha}\in \iota\circ \Lambda\circ \iota^{-1}$ such that $\bar{\alpha}(\bar{\xi})=\iota(\xi)$. Replacing $\iota$ by $\bar{\alpha}\circ \iota$ allows us to assume $\iota(\xi)=\bar{\xi}$. 
Since $\div(\bar{\rho})_\infty=(\frac{2q}{p})\bar{P}_\infty$, and that number is the second smallest non-gap at $\bar{P}_\infty$, we also have $\iota(\rho)=\kappa_1\bar{\xi}+ \kappa\bar{\rho}$ with $\kappa,\kappa_1\in \mathbb{F}_{q^2}.$ This shows that $\iota$ is linear, so that it extends to a linear map of the affine planes containing $\cC$ and $\Bar{\cC}$, respectively. Therefore, the tangent line to $\cC$ at $O$ which has equation $Y=0$ is mapped to the tangent line to $\Bar{\cC}$ at $\bar{O}$. This yields $\kappa_1=0$, i.e. $\iota(\rho)=\kappa\bar{\rho}$ with $\kappa\in \mathbb{F}_{q^2}^*$. We show that $\kappa\in \mathbb{F}_p^*$ using a previous argument: $\bar{\rho}$ and $\iota(\rho)$ have the same zeros (each counted with multiplicity $2$), namely the places centered at the $\frac{q}{p}$ (non-singular) points $P=(t,0)$ of $\Bar{\cC}$. Thus $Tr(t)=0$ and $Tr(kt)=0$ hold simultaneously whence $\kappa\in \mathbb{F}_p^*$ follows.

Now, $Tr(\iota(\xi))^2-2b Tr(\iota(\rho))=0$. Thus $Tr(\bar{\xi})-2bTr(\kappa\bar{\rho})=0$ whence $Tr(\bar{\xi})^2-2b\kappa Tr(\bar{\rho})=0$ follows. On the other hand, $Tr(\bar{\xi})^2-2\bar{b}Tr(\bar{\rho})=0$. 
Comparison shows that $\bar{b}=\kappa b$. Thus  $\bar{b}=\kappa b$ with $\kappa \in \mathbb{F}_p^*$ is a necessary condition for the existence of an $\mathbb{F}_{q^2}$-rational isomorphism between  $G_b$ and $G_{\bar{b}}$. This condition is also sufficient.  
\end{proof}

\section{Galois subcovers of $\mathbb{F}_{q^2}(\cH)$ with respect to a cyclic subgroup of order $4$ and $p=2$}
Take $\mathbb{F}_{q^2}(\cH_q)$ in its canonical form $\mathbb{F}_{q^2}(x,y)$ with $y^q+y+x^{q+1}=0$. Fix $b\in \mathbb{F}_{q^2}$ such that $b^q+b+1=0$. Then $\psi_{1,b,1}$ has order $4$ and its square is $\psi_{0,1,1}$. Let  $\eta=y^2+y$. Then $\psi_{0,1,1}(\eta)=\eta$. Since $\psi_{0,1,1}(x)=x$, this yields that 
the Galois subcover $\mathbb{F}_{q^2}(\cF')$ of $\mathbb{F}_{q^2}(\cH_q)$ is $\mathbb{F}_{q^2}(x,\eta)$ with 
\begin{equation}\label{Y200423}
 x^{q+1}+\sum_{i=0}^{h-1}\eta^{2^i}=0.
\end{equation}
In particular, $\mathbb{F}_{q^2}(\cF')$ is the same as in (i) of Result \ref{ckt} with $\omega=1$. 
We go on by finding another form for $\mathbb{F}_{q^2}(\cF')$. For this purpose, let $\xi=x^2+x$, and $F$ the fixed field of $\langle \psi_{0,1,1}\rangle$. We prove that 
$\mathbb{F}_{q^2}(\cF')=\mathbb{F}_{q^2}(\xi,\eta)$. Since $\psi_{0,1,1}(\eta)=\eta$ and 
$\psi_{0,1,1}(\xi)=\psi_{0,1,1}(x^2+x)=\psi_{0,1,1}(x)^2+\psi_{0,1,1}(x)=x^2+x=\xi,$
we have $\mathbb{F}_{q^2}(\xi,\eta)\subseteq F$. Moreover, $[\mathbb{F}_{q^2}(\mathcal{F}'):\mathbb{F}_{q^2}(\xi,\eta)]=2$. From this either $\mathbb{F}_{q^2}(\xi,\eta)=F$ or $F=\mathbb{F}_{q^2}(\mathcal{F}')$. The latter case cannot actually occur, and hence $F=\mathbb{F}_{q^2}(\xi,\eta)$. We are going to eliminate $x$ from (\ref{Y200423}) and $\xi=x^2+x$. Since
\begin{equation}
    \label{trsigma}
   \sum_{i=0}^{h-1}\xi^{2^i} =x^q+x,
\end{equation}
and the $(q+1)$-power of $\xi=x^2+x$ is
\begin{equation}
    \label{potinv260423}
    \xi^{q+1}=x^{2(q+1)}+x^{q+1}+x^{q+1}(x^q+x),
\end{equation}
using (\ref{Y200423}) and (\ref{trsigma}) 
\begin{equation}\label{prop240323}
Tr(\eta)^2+Tr(\eta)+Tr(\eta)Tr(\xi)=\xi^{q+1}
\end{equation}
follows.

Since $\psi_{1,b,1}$ commutes with $\psi_{0,1,1}$, $\psi_{1,b,1}$ induces an $\mathbb{F}_{q^2}$-automorphism of $\mathbb{F}_{q^2}(\cF')$. A straightforward computation shows that $\psi_{1,b,1}(\xi)=\xi$ and $\psi_{1,b,1}(\eta)=\eta+\xi+b^2+b$. Let $\Phi_b=\langle\psi_{1,b,1}\rangle.$
\begin{theorem}
\label{the25052023}
The Galois subcover $G_b$ of $\mathbb{F}_{q^2}(\mathcal{F}')$ with respect to $\Phi_b=\langle\psi_{1,b,1}\rangle$ has the following properties:
\begin{itemize}
\item[(I)]
$G_b=\mathbb{F}_{q^2}(\xi,\kappa)$ with
\begin{equation}\label{prop240323A}
\alpha_0(\xi)+\alpha_1(\xi)\kappa+\ldots+\alpha_i(\xi)\kappa^{2^i}+\ldots+ \alpha_{h-1}(\xi)\kappa^{2^{h-1}}=0
\end{equation}
where $\alpha_i(\xi)\in \mathbb{F}_{q^2}[\xi]$ and
$$\alpha_0(\xi)=\xi^{q+1},\,\, \alpha_1(\xi)= \frac{(\xi^q+\xi)^2+(\xi+b+b^2)^q(\xi^q+\xi)}{Tr(\xi)},\,\,\alpha_h(\xi)=(\xi+b+b^2)^{2q}$$
and the other coefficients $\alpha_i(\xi)$ for $2\le i \le h-2$ are computed recursively from the equation $$Tr(\xi)\alpha_i(\xi)+\alpha_{i-1}(\xi)^2+(\xi+b+b^2)^{2q}+(\xi+b+b^2)^q(\xi^q+\xi)$$
where $Tr(\xi)=\xi+\xi^p+\ldots+\xi^{\nicefrac{q}{p}}$ and $q=2^h$.
\item[(II)]
The genus of $G_b$ equals $p^{h-3}(p^h-p)=\frac{1}{8}q(q-2)$.
\item[(III)] The $\mathbb{F}_{q^2}$-automorphism group of $G_b$ has order $\ha {q^2}$.
\end{itemize}
\end{theorem}
\begin{proof}
(I) Let
\begin{equation}
\label{eqD050123} \zeta=\eta^2+\eta(\xi+b+b^2).
\end{equation}
Then $\zeta$ is an $\psi_{1,b,1}$-invariant element of $\mathbb{F}_{q^2 }(\bar{\mathcal{F}'})$. In fact,
$\psi_{1,b,1}^2(\eta)=\eta$ and $\zeta=\eta\psi_{1,b,1}(\eta)$.
To obtain an equation for $G_b$, it is enough to eliminate $\eta$ from (\ref{eqD050123}) and (\ref{prop240323}).
From (\ref{eqD050123}) 
\begin{equation}
    \label{formcanzeta}
    \Big(\frac{\eta}{\xi+b+b^2}\Big)^2+\frac{\eta}{\xi+b+b^2}+\frac{\zeta}{(\xi+b+b^2)^2}=0.
\end{equation}
Let $$\kappa=\frac{\zeta}{(\xi+b+b^2)^2}.$$
Since the trace of the expression on the left hand side in (\ref{formcanzeta}) is also zero, it follows
$$\eta^q+(\xi+b+b^2)^{q-1}\eta+Tr(\kappa)(\xi+b+b^2)^{q}=0,$$
which yields \begin{equation}
  \label{tracformcan} \eta^q+A\eta+B=0.
\end{equation}
with
$$A=(\xi+b+b^2)^{q-1}, \quad B=Tr(\kappa)(\xi+b+b^2)^q.$$
In (\ref{prop240323}), $Tr(\eta)^2+Tr(\eta)$ can be replaced by $\eta^q+\eta$. Therefore,
\begin{equation}
    \label{150523a}    \eta^q+\eta+Tr(\eta)Tr(\xi)+\xi^{q+1}=0.
\end{equation}
Summing (\ref{tracformcan}) and (\ref{150523a}) gives
\begin{equation}
    \label{150523b}
    (A+1)\eta+B+Tr(\eta)Tr(\xi)+\xi^{q+1}=0
\end{equation}
whence
\begin{equation}
    \label{treta150523}
    Tr(\eta)=\frac{(A+1)\eta+B+\xi^{q+1}}{Tr(\xi)}.
\end{equation}
Substituting (\ref{treta150523}) in (\ref{prop240323}) gives
\begin{equation}
    \label{150523c1}
    \frac{(A+1)^2\eta^2+B^2+\xi^{2(q+1)}}{Tr(\xi)^2}+\frac{(A+1)\eta+B+\xi^{q+1}}{Tr(\xi)}+(A+1)\eta+B=0.
\end{equation}
Clearing the denominators yields
\begin{equation}
    \label{150523c}
    (A+1)^2\eta^2+B^2+\xi^{2(q+1)}+((A+1)\eta+B+\xi^{q+1})Tr(\xi)+((A+1)\eta+B) Tr(\xi)^2=0,
\end{equation}
whence
\begin{equation}
    \label{150523d}
    \eta^2+\frac{Tr(\xi)+Tr(\xi)^2}{A+1}\eta+\frac{B^2+\xi^{2(q+1)}+B(\xi+\xi^q)+\xi^{q+1}Tr(\xi)}{(A+1)^2}=0.
\end{equation}
Now, since $Tr(\xi)+Tr(\xi)^2=(A+1)(\xi+b+b^2)$, summing (\ref{eqD050123}) and (\ref{150523d}) gives
\begin{equation}
    \label{150523e}    \kappa(\xi+b+b^2)^2+\frac{B^2+\xi^{2(q+1)}+B(\xi+\xi^q)+\xi^{q+1}Tr(\xi)}{(A+1)^2}=0,
\end{equation}
whence
\begin{equation}
    \label{150523f}    \kappa(\xi^q+\xi)^2+B^2+\xi^{2(q+1)}+B(\xi+\xi^q)+\xi^{q+1}Tr(\xi)=0.
\end{equation}
Therefore,
\begin{equation}
    \label{eqfinal}
     \kappa(\xi+\xi^q)^2+(\xi+b+b^2)^{2q}Tr(\kappa)^2+(\xi+b+b^2)^q(\xi+\xi^q)Tr(\kappa)+\xi^{2(q+1)}+\xi^{q+1}Tr(\xi)=0.
\end{equation}
From Lemma \ref{lem210523}, there exists a polynomial $G(X,Y)=\alpha_0(X)+\alpha_1(X)Y+\ldots+\alpha_{h}(X)Y^{2^h}\in \mathbb{F}_{q^2}[X,Y]$ such that either $G(\xi,\kappa)=0$ or $G(\xi,\kappa)+Tr(\xi)=0$.
We may assume that the former case occurs. If $G(X,Y)$ is reducible then it has an irreducible factor in $\mathbb{F}[X,Y]$ whose degree in $Y$ is  at most $\ha q$. Then $[\mathbb{F}(\xi,\kappa):\mathbb{F}(\xi)]\le\ha q$. On the other hand $[\mathbb{F}_{q^2}(\xi,\eta):\mathbb{F}_{q^2}(\xi,\kappa)]=2$ and $[\mathbb{F}_{q^2}(\xi,\eta):\mathbb{F}_{q^2}(\xi)]=q$, and hence $[\mathbb{F}_{q^2}(\xi,\kappa):\mathbb{F}_{q^2}(\xi)]=\ha q$. From this,
$[\mathbb{F}(\xi,\kappa):\mathbb{F}(\xi)]=\ha q$. Therefore $G(X,Y)$ is irreducible over $\mathbb{F}$. Thus $G(X,Y)=0$ is an equation for $G_b$.

(II) The claim follows from Result \ref{gsx1}. 

(III)
For every $a\in \mathbb{F}_q$ and $c\in \mathbb{F}_{q^2}$ with $c^q+c+a^{q+1}=0$, the map
$\psi_{a,c}$ given by
$$
\begin{cases}
x':=x+a,\\
\eta':=\eta+a^{2q}x^2+a^qx+c+c^2
\end{cases}
$$
is an $\mathbb{F}_{q^2}$-automorphism of $\mathbb{F}_{q^2}(\cF')$. In fact, $Tr(\eta')+x'^{q+1}=Tr(\eta)+a^{q^2}x^q+a^qx+c^q+c+(x+a)^{q+1}=Tr(\eta)+ax^q+a^qx+a^{q+1}+x^{q+1}+ax^q+a^qx+a^{q+1}=Tr(\eta)+x^{q+1}$. Moreover, $\psi_{a,c}=\psi_{a',c'}$ if and only if $a=a'$ and either $c=c'$ or $c'=c+1$. Therefore the $\mathbb{F}_{q^2}$-automorphism group $\Psi$ of $\mathcal{F}'$ consisting of all $\psi_{a,c}$ is a group of order $\ha q^2$.

Since $\xi=x^2+x$, $Tr(\xi)=x^q+x$. 
Moreover, we prove that
\begin{equation}
\label{eqA010523} x=\frac{Tr(\eta)+\xi}{Tr(\xi)+1}.
\end{equation}
 For this purpose, we show
\begin{equation}
\label{eqB010523} \xi=\Big(\frac{Tr(\eta)+\xi}{Tr(\xi)+1}\Big)^2+\frac{Tr(\eta)+\xi}{Tr(\xi)+1}.
\end{equation}
From (\ref{prop240323}), the right hand side in (\ref{eqB010523}) equals
$$\frac{(Tr(\eta)+\xi)^2+(Tr(\eta)+\xi)(Tr(\xi)+1)}{(Tr(\xi)+1)^2}=\frac{\xi^{q+1}+\xi^2+\xi(Tr(\xi)+1)}{Tr(\xi)^2+1}=\frac{\xi^{q+1}+\xi+\xi^3+\ldots+\xi^{\nicefrac{q}{2}+1}}{1+\xi^2+\xi^{q/2}+\xi^q}=\xi $$
whence (\ref{eqB010523}) follows. Since $\xi=x^2+x$, (\ref{eqB010523}) yields either (\ref{eqA010523}), or
$$x=\frac{Tr(\eta)+\xi}{Tr(\xi)+1}+1.$$
The latter case cannot actually occur as $O=(0,0)$ is a non-singular point of $\mathcal{F}'$ and $x_O=\eta_O=\xi_O=0$.
Therefore, $\mathbb{F}_{q^2}(\mathcal{F}')=\mathbb{F}_{q^2}(\xi,\eta)$ with (\ref{prop240323}).
In terms of $\xi$ and $\eta$, the  $\mathbb{F}_{q^2}$-automorphism $\psi_{a,c}$ is given by
$$ (\xi,\eta) \mapsto \Big(\xi+a^2+a,\eta+a^{2q}\Big(\frac{Tr(\eta)+\xi}{Tr(\xi)+1}\Big)^2+a^q\frac{Tr(\eta)+\xi}{Tr(\xi)+1}+c+c^2\Big).
$$
From (\ref{eqB010523}), $\psi_{1,b}=\psi_{1,b,1}$.  
Let $\Psi_{1,b}$ denote the subgroup of $\Psi$ generated by $\psi_{1,b,1}$. A straightforward computation shows that the normalizer $N_\Psi(\Psi_{1,b})$ of $\Psi_{1,b}$ in $\Psi$ consists of all elements $\psi_{a,c}$ with $a\in \mathbb{F}_q$, or $a^q+a=1$. Thus, $N_\Psi(\Psi_{1,b})$ is an $\mathbb{F}_{q^2}$-automorphism group of $\mathcal{F}'$ of order $q^2$.

$G_b$ is the Galois subcover of $\mathbb{F}_{q^2}(\cF')$ with respect to $\Psi_{1,b}$. From Result \ref{gsx1}, $G_b$ has genus $\frac{1}{8}q(q-2)$. In fact, by construction, $G_b$ is the quotient curve of the Hermitian curve of equation $y^q+y+x^{q+1}=0$ with respect to its $\mathbb{F}_{q^2}$-automorphism group of order $4$ generated by $(x,y)\mapsto (x+1,y+x+b)$ with $b^q+b+1=0$. Moreover, $N_{\Psi}(\Psi_{1,b})$ induces on $G_b$ an $\mathbb{F}_{q^2}$-automorphism group $\bar{\Psi}$ of order $\ha q^2$. Let $\bar{\Psi}=\Psi/\Psi_{1,b}$. 
Since $G_b$ has as many as $\frac{1}{4}q^2(q+2)+1$ $\mathbb{F}_{q^2}$-rational places, and $\bar{\Psi}$ fixes one of these places and acts on the remaining $\frac{1}{4}q^2(q+2)=\frac{1}{2}q^2(\frac{1}{2}q+1)$ ones, it turns out that
$\bar{\Psi}$ is the (unique) $2$-Sylow subgroup of $G_b$. Therefore, since $|\bar{\Psi}|=\frac{1}{2}q^2>\frac{1}{8}q(q-2)=\mathfrak{g}(G_b)$, the subcover $\hat{G_b}=G_b/\bar{\Psi}$
is rational; see Result \ref{sti1}. Under the $\frac{1}{4}q^2(q+2)$ $\mathbb{F}_{q^2}$-rational places of $G_b$ there are exactly $\ha q+1$ places of $\hat{G_b}$. 
Now, let $\alpha$ be an $\mathbb{F}_{q^2}$-automorphism of $G_b$ whose order is an odd prime. Then $\alpha$ induces an $\mathbb{F}_{q^2}$-automorphism $\hat{\alpha}$ of  $\hat{G_b}$ which preserves the set of the above $\frac{1}{2}q+1$ places. Since $\alpha$ and $\hat{\alpha}$ have the same odd  order, $\hat{\alpha}$ fixes two $\mathbb{F}_{q^2}$-rational places of $\hat{G_b}$ and acts fixed-point-free on the set of its remaining $\ha q$ places.
But this cannot actually occur as $\bar{\alpha}$ has odd prime order. This completes the proof of (III).
\end{proof}

\section{Appendix}\begin{lemma}
\label{lem060423} Any irreducible factor over $\mathbb{F}_p$ of the polynomial
 \begin{equation}
     \label{eqA030423} F(X,Y)= (Y-Y^{p^3})(Y-Y^p)^p(X-X^{p^2})^{p+1}- (X-X^{p^3})(X-X^p)^p(Y-Y^{p^2})^{p+1}.
 \end{equation}
 is either linear $\alpha X+ \beta Y+ \gamma$  or quadratic $XY+\alpha X+ \beta Y+ \gamma$. In the former case,  $\alpha\in \{0,1\}$  and $\alpha, \beta$ do not vanish simultaneously, in the latter case $\alpha\beta\neq \gamma$.
\end{lemma}
\begin{proof} Observe that $\deg(F(X,Y))\le 2p^3+p^2+p$. The linear substitution $(X,Y)\mapsto (\alpha_1X+\alpha_2,\beta_1Y+\beta_2)$ with $\alpha_1,\alpha_2\in \mathbb{F}_p^*$ and $\beta_1,\beta_2\in \mathbb{F}_p$  transforms $F(X,Y)$ to $(\alpha_1\beta_1)^2 F(X,Y)$. Since $X$ is a factor of $F(X,Y)$ of multiplicity $p+1$, this shows that $X-\gamma$ with $\gamma\in \mathbb{F}_p$ is also a factor of $F(X,Y)$ with multiplicity $p+1$. Similarly, $Y-\gamma$ with $\gamma\in \mathbb{F}_p$ is a factor of $F(X,Y)$ with multiplicity $p+1$. This produces $2p$ linear factors of $F(X,Y)$ each of multiplicity $p+1$. Since $Y-X$ is a linear factor of $F(X,Y)$ the same argument shows that $\alpha X+ \beta Y+ \gamma$ is a linear factor of $F(X,Y)$ whenever  $\alpha\in \{0,1\}$  and $\alpha, \beta$ do not vanish simultaneously. This produces $p^2-p$ more linear factors of $F(X,Y)$. Also, for any $\gamma\in \mathbb{F}_p^*$, we have that $XY-\gamma$ is a (quadratic) factor of $F(X,Y)$. The above linear substitution  with $\alpha_1=\alpha_2=1$ takes the hyperbole $XY-\gamma$ to  $XY+\beta_2X+\beta_1Y+\beta_1\beta_2\gamma$. As many as $p^2(p-1)$ such quadratic factors of $F(X,Y)$ are obtained in this way. Since $\deg(F(X,Y))\le 2p^3+p^2+p =2p(p+1)+p^2-p+2(p^3-p)=2p^3+p^2+p$, this yields $\deg(F(X,Y))=2p^3+p^2+p$. Therefore, no further irreducible factor of $F(X,Y)$ over $\mathbb{F}_p$ exists.
\end{proof}
\begin{lemma}
\label{lem210523} Let $p=2$ and $b\in \mathbb{F}_q$ such that $b+b^q+1=0$. Then the polynomial
$$F(X,Y)=Y(X+X^q)^2+(X+b+b^2)^{2q}Tr(Y)^2+(X+b+b^2)^q(X+X^q)Tr(Y)+X^{2(q+1)}+X^{q+1}Tr(X)$$
factorizes into $G(X,Y)(G(X,Y)+Tr(X))$ with $$G(X,Y)=\alpha_0(X)+\alpha_1(X)Y+\ldots+\alpha_i(X)Y^{2^i}+\ldots+ \alpha_{h-1}(X)Y^{2^{h-1}}$$
where
$$\alpha_0(X)=X^{q+1},\,\, \alpha_1(X)= \frac{(X^q+X)^2+(X+b+b^2)^q(X^q+X)}{Tr(X)},\,\,\alpha_h(X)=(X+b+b^2)^{2q}$$
and the other coefficients $\alpha_i(X)$ for $2\le i \le h-2$ are computed recursively from the equation $$Tr(X)\alpha_i(X)+\alpha_{i-1}(X)^2+(X+b+b^2)^{2q}+(X+b+b^2)^q(X^q+X).$$
\end{lemma}
\begin{proof} A straightforward computation shows that $F(X,Y)=X^{2(q+1)}+X^{q+1}Tr(X)+Y(\ldots)$. Therefore, $X^{2(q+1)}+X^{q+1}Tr(X)=\alpha_0(X)Tr(X)+\alpha_0(X)^2$. Comparison yields $\alpha_0(X)= X^{q+1}$. Furthermore,
the coefficient of $Y$ in $F(X,Y)$ equals $(X^q+X)^2+(X+b+b^2)^q(X+X^q)$ whence  $(X+X^q)^2+(X+b+b^q)^q(X+X^q)=\alpha_1(X)Tr(X)$. Also, $(X+b+b^q)^{2q}$ is the coefficient of $Y^q$ in $F(X,Y)$. Thus   $(X+b+b^q)^{2q}=\alpha_{h-1}(X)^2$ whence $\alpha_{h-1}(X)=(X+b+b^q)^q$. The other equations follow by comparison of the coefficient of $Y^{2^i}$ for $2\le i \le h-2$ in $F(X,Y)$ and $G(X,Y)(G(X,Y)+Tr(X))$.
\end{proof}

\end{document}